\documentclass[12pt,reqno]{amsart}

\newcommand\version{January 12, 2016}

\usepackage{amsmath,amsfonts,amsthm,amssymb,amsxtra}

\newtheorem{theorem}{Theorem}[section]
\newtheorem{proposition}[theorem]{Proposition}
\newtheorem{lemma}[theorem]{Lemma}
\newtheorem{corollary}[theorem]{Corollary}

\theoremstyle{definition}

\theoremstyle{remark}

\numberwithin{equation}{section}

\newcommand{\C}{\mathbb{C}}

\renewcommand{\epsilon}{\varepsilon}

\newcommand{\loc}{{\rm loc}}

\newcommand{\N}{\mathbb{N}}

\renewcommand{\phi}{\varphi}
\newcommand{\R}{\mathbb{R}}

\DeclareMathOperator{\dom}{dom}
\DeclareMathOperator{\im}{Im}

\DeclareMathOperator{\re}{Re}

\DeclareMathOperator{\tr}{Tr}

\begin{document}

\title[On the number of eigenvalues --- \version]{On the number of eigenvalues of Schr\"odinger operators with complex potentials}

\author{Rupert L. Frank}
\address{Rupert L. Frank, Mathematics 253-37, Caltech, Pasadena, CA 91125, USA}
\email{rlfrank@caltech.edu}

\author{Ari Laptev}
\address{Ari Laptev, Department of Mathematics, Imperial College London, SW7 2AZ, London, UK, and Institut Mittag-Leffler, Djursholm, Sweden}
\email{laptev@mittag-leffler.se}

\author{Oleg Safronov}
\address{Oleg Safronov, Department of Mathematics	 and Statistics, University of North Carolina at Charlotte, Charlotte, NC 28223, USA}
\email{osafrono@uncc.edu}

\begin{abstract}
We study the eigenvalues of Schr\"odinger operators with complex potentials in odd space dimensions. We obtain bounds on the total number of eigenvalues in the case where $V$ decays exponentially at infinity.
\end{abstract}


\maketitle

\renewcommand{\thefootnote}{${}$} \footnotetext{\copyright\, 2016 by the authors. This paper may be reproduced, in its entirety, for non-commercial purposes.}


\section{Introduction and main results}

 Let $V$   be a complex-valued  potential on ${\Bbb R}^d$, where $d$ is  odd.
 We  study  the spectral properties of the  Schr\"odinger operator 
\begin{equation}\label{Hpic}
-\Delta+V(x).
\end{equation}
 Namely,  denote
by $\lambda_j$ the eigenvalues of   the operator \eqref{Hpic}. 
We are interested in an estimate of the total number $N$ of the  eigenvalues $\lambda_j$ in the case  where $V$  decays exponentially fast.

There has been a lot of recent activity concerning uniform bounds on eigenvalues of Schr\"odinger operators with complex-valued potentials which are decaying at infinity. By `uniform bounds' we refer to bounds which do not only hold in an asymptotic regime and which depend on the potential only through some simple and computationally easily accessible quantities like $L^p$ norms. We refer to \cite{Da} for a review of the state of the art of non-selfadjoint Schr\"odinger operators and for motivations and applications. Bounds on single eigenvalues were proved, for instance, in \cite{AAD,DN,FrLaSe,Fr,En} and bounds on sums of powers of eigenvalues were proved, for instance, in \cite{FrLaLiSe,LaSa,DeHaKa0,DeHaKa,BGK,FrSa,Fr3}. The latter bounds generalize the Lieb--Thirring bounds \cite{LiTh} to the non-selfadjoint setting.

Despite this activity, there have been almost no results on the \emph{number} of eigenvalues of Schr\"odinger operators with complex potentials. (The only exceptions are two papers \cite{St1,St3} in one and three dimensions, whose relation to our work we discuss below.) To see that this is a subtle question we recall that, for instance, for the Schr\"odinger operator $-d^2/dx^2+V$ on the half-line with a Dirichlet boundary condition at the origin, Bargmann's bound states that for \emph{real} potentials $V$ the number of eigenvalues can be bounded by $\int |x| |V(x)|\,dx$. It is a remarkable result of Pavlov \cite{Pa2} that a similar bound cannot hold in the non-selfadjoint case. In fact, he showed that for any $0<\alpha<1/2$ and any $\lambda>0$ there is a (real) potential $V$ satisfying, for some $C,c>0$,
\begin{equation}
\label{eq:pavlov}
|V(x)| \leq C e^{-cx^\alpha}
\qquad\text{for all}\ x\in (0,\infty)
\end{equation}
and a complex number $\sigma$ such that the operator $-d^2/dx^2+V$ in $L^2(0,\infty)$ with boundary condition $\psi'(0)=\sigma\psi(0)$ has an infinite number of eigenvalues accumulating at $\lambda$. On the other hand, Pavlov \cite{Pa1} also showed that if, for some $C,c>0$,
\begin{equation}
\label{eq:pavlov2}
|V(x)| \leq C e^{-cx^{1/2}}
\qquad\text{for all}\ x\in (0,\infty) \,,
\end{equation}
then the number of eigenvalues of the operator $-d^2/dx^2+V$ in $L^2(0,\infty)$ with any boundary condition of the form $\psi'(0)=\sigma\psi(0)$, $\sigma\in\C$, or $\psi(0)=0$ is finite. Pavlov also proves a similar theorem in three dimensions. Pavlov's proofs, however, seem to give no bound on the number of eigenvalues in terms of the constants $c$ and $C$ in \eqref{eq:pavlov2}.

Before Pavlov, Naimark \cite{Na} had shown that the number of eigenvalues is finite if \eqref{eq:pavlov} holds with $\alpha=1$. Similar results are known in two and three dimensions, see \cite{Ma1,Ma2,Mu}, but none of these proofs gives uniform bounds on the corresponding number of eigenvalues. Such uniform bounds, in arbitrary odd dimensions, are the main result of the present paper. More precisely, we shall prove the following two theorems.

\begin{theorem}
\label{main1}
The number $N$ of eigenvalues of $-\frac{d^2}{dx^2}+V$ in $L^2(\R_+)$ with a Dirichlet boundary condition, counting algebraic multiplicities, satisfies, for any $\epsilon>0$,
$$
N \leq \frac{1}{\epsilon^2} \left( \int_0^\infty e^{\epsilon x} |V(x)|\,dx \right)^2 \,.
$$
\end{theorem}

\begin{theorem}
\label{main2}
Let $d\geq 3$ be odd. Then the number $N$ of eigenvalues of $-\Delta+V$ in $L^2(\R^d)$, counting algebraic multiplicities, satisfies, for any $\epsilon>0$,
$$
N \leq \frac{C_d}{\epsilon^2} \left( \int_{\R^d} e^{\epsilon |x|} |V(x)|^{(d+1)/2} \,dx \right)^2
$$
with a constant $C_d$ depending only on $d$.
\end{theorem}

The proofs of both theorems are based on a `trace formula approach' which consists in identifying eigenvalues with zeroes of a certain analytic function and in using bounds on the zeroes of analytic functions. This approach was used before by two of us (A.L. and O.S.) in self-adjoint problems \cite{LNS} and by one of us (O.S.) in non-selfadjoint problems \cite{Sa}. In non-selfadjoint problems, a related method is used, for instance, in \cite{BGK,DeHaKa0,DeHaKa,FrSa}. In the present paper we combine these techniques with novel resolvent bounds in trace ideals, which are our technical main results. Resolvent bounds in operator norm are due to Kenig--Ruiz--Sogge \cite{KRS}. In connection to eigenvalue bounds for non-selfadjoint operators they were exploited in \cite{Fr} and generalized to trace ideals in \cite{FrSa}. Here we go a significant step further and show that, if $V$ decays exponentially in the sense of the assumptions in Theorems \ref{main1} and \ref{main2}, then the Birman--Schwinger operator admits an analytic continuation and resolvent bounds, similar to those of \cite{FrSa}, remain valid for its continuation. Our proof uses complex interpolation as in \cite{KRS} and \cite{FrSa}, but the choice of the analytic family is more involved than in those papers.

In dimensions one and three the resolvent kernel is explicit and this is important for the proofs in \cite{St1,St3}. In contrast, our Theorem \ref{main2} is valid in arbitrary odd dimensions $d\geq 3$, where the resolvent kernel is only given in terms of Bessel functions, which become increasingly more complicated as the dimension increases. Complex interpolation helps us to go around this obstacle. The assumption that the space dimension is odd comes from the fact that in this case the resolvent admits an analytic continuation to the lower half-plane (while there is a branch point at zero for even dimensions).

Finally, we note that, while we managed to obtain rather explicit and transparent bounds for potentials decaying exponentially, the question whether there is a quantitative version of Pavlov's bound remains a challenging open question.

\subsection*{Acknowledgements}

The first and third author would like to thank the Mittag--Leffler Institute for hospitality. The first author acknowledges support through NSF grant DMS-1363432.
AL was supported by the grant of the Russian Federation Government to support scientific
research under the supervision of leading scientist at Siberian Federal University, No 14.Y26.31.0006.


\section{Zeroes of analytic functions}

The following proposition gives a useful bound on the zeroes of an analytic function in a half-plane.

\begin{proposition}\label{zeroes}
Let $\eta\in\R\setminus\{0\}$. Let $a$ be an analytic function in $\{\im k>\eta\}$ which is continuous up to the boundary and satisfies
\begin{equation}
\label{eq:ass1}
a(k) = 1 + o(|k|^{-1})
\qquad\text{as}\ |k|\to\infty \ \text{in}\ \{\im k>\eta\}
\end{equation}
and, for some $A\geq 0$ and $\nu>1$,
\begin{equation}
\label{eq:ass2}
\ln|a(k)| \leq A |k|^{-\nu}
\qquad\text{if}\ \im k = \eta \,.
\end{equation}
Then the zeroes $k_j$ of $a$ in $\{\im k>\eta\}$, repeated according to multiplicities, satisfy
\begin{equation}
\label{eq:zeroes}
\sum_j \left( \im k_j - \eta \right) \leq c_\nu A |\eta|^{-\nu+1}
\end{equation}
with
$$
c_\nu = \frac{1}{2\pi} \int_\R \frac{dt}{(1+t^2)^{\nu/2}} \,.
$$
\end{proposition}

The integral appearing in $c_\nu$ can be expressed in terms of the Gamma function. For $\nu=2$, the computation is straightforward and we obtain
\begin{equation}
\label{eq:zeroesconst1}
c_2 = 1/2 \,.
\end{equation}

\begin{proof}
We introduce the Blaschke product
$$
B(k) = \prod_j \frac{k-k_j}{k-\overline{k_j} - 2i\eta} \,,
$$
so that $a(k)/B(k)$ is an analytic and non-zero in $\{\im k>\eta\}$ and $\log (a(k)/B(k))$ exists and is analytic there. For $R>\eta$ we denote by $C_R$ the contour which consists of the interval $\{k\in {\Bbb C}:\,  k=x+i\eta,\ |x|\leq \sqrt{R^2-\eta^2}\}$, traversed from left to right, and the circular part $\Gamma_R:=\{k\in {\Bbb C}:\, |k|=R \,, \im k >\eta\}$, traversed counterclockwise. 

Then
$$
\int_{C_R} \log \frac{a(k)}{B(k)}\,dk = 0 \,,
$$
and therefore
\begin{equation}
\label{eq:zeroesproof1}
\re \int_{-\sqrt{R^2-\eta^2}}^{\sqrt{R^2-\eta^2}} \log \frac{a(x+i\eta)}{B(x+i\eta)}\,dx + \re \int_{\Gamma_R} \log\frac{a(k)}{B(k)}\,dk = 0 \,.
\end{equation}
We note that $|B(x+i\eta)|=1$ if $x\in\R$ and, therefore,
\begin{align}
\label{eq:zeroesproof2}
\re \int_{-\sqrt{R^2-\eta^2}}^{\sqrt{R^2-\eta^2}} \log \frac{a(x+i\eta)}{B(x+i\eta)}\,dx 
& = \int_{-\sqrt{R^2-\eta^2}}^{\sqrt{R^2-\eta^2}} \ln \left| \frac{a(x+i\eta)}{B(x+i\eta)} \right| \,dx \notag \\
& = \int_{-\sqrt{R^2-\eta^2}}^{\sqrt{R^2-\eta^2}} \ln \left| a(x+i\eta) \right| \,dx \,.
\end{align}
On the other hand, by \eqref{eq:ass1} and $B(k) = 1+O(|k|^{-1})$ (note that the $|k_j|$ are contained in a bounded set as a consequence of \eqref{eq:ass1}), both $\log a(k)$ and $\log B(k)$ are well-defined for all sufficiently large $|k|$ and we have, for all sufficiently large $R$,
\begin{equation}
\label{eq:zeroesproof3}
\re \int_{\Gamma_R} \log\frac{a(k)}{B(k)}\,dk = \re \int_{\Gamma_R} \log a(k)\,dk - \re \int_{\Gamma_R} \log B(k)\,dk \,.
\end{equation}
We conclude from \eqref{eq:zeroesproof1}, \eqref{eq:zeroesproof2} and \eqref{eq:zeroesproof3} that
\begin{equation}
\label{eq:traceformula}
\re \int_{\Gamma_R} \log B(k)\,dk = \int_{-\sqrt{R^2-\eta^2}}^{\sqrt{R^2-\eta^2}} \ln \left| a(x+i\eta) \right| \,dx + \re \int_{\Gamma_R} \log a(k)\,dk 
\end{equation}
for all sufficiently large $R$.

We assume that $R$ is so large that $|k_j|<R$ and $|k_j -i\eta|<\sqrt{R^2-\eta^2}$ for all $j$. Then, by analyticity,
\begin{equation}
\label{eq:tf0}
\int_{\Gamma_R} \log B(k)\,dk = \int_{\tilde\Gamma_R} \log B(k) \,dk
\end{equation}
with $\tilde\Gamma_R := \{ |k-i\eta|=\sqrt{R^2-\eta^2}\,, \im k>\eta \}$, traversed counterclockwise. Since
$$
\log B(k) = 2i \sum_j \frac{\eta - \im k_j}{k-i\eta} + O( (k-i\eta)^{-2}) \,,
$$
we get
\begin{align}
\label{eq:tf1}
\int_{\tilde\Gamma_R} \log B(k) \,dk & = \int_{|\tilde k|=\sqrt{R^2-\eta^2},\ \im\tilde k>0} \log B(\tilde k+i\eta) \,d\tilde k \notag \\
& = -2\pi \sum_j \left(\eta - \im k_j\right) + O( (R^2-\eta^2)^{-1/2})
\qquad\text{as}\ R\to\infty \,.
\end{align}
On the other hand, by \eqref{eq:ass1},
\begin{align}
\label{eq:tf2}
\re \int_{\Gamma_R} \log a(k)\,dk = o(1)
\qquad\text{as}\ R\to\infty \,.
\end{align}
Finally, by \eqref{eq:ass2},
\begin{align}
\label{eq:tf3}
\int_{-\sqrt{R^2-\eta^2}}^{\sqrt{R^2-\eta^2}} \ln \left| a(x+i\eta) \right| \,dx & \leq A \int_{-\sqrt{R^2-\eta^2}}^{\sqrt{R^2-\eta^2}} \frac{dx}{(x^2+\eta^2)^{\nu/2}} \leq A \int_\R \frac{dx}{(x^2+\eta^2)^{\nu/2}} \notag \\
& = A |\eta|^{-\nu+1}  \int_\R \frac{dt}{(1+t^2)^{\nu/2}} \,.
\end{align}
Inequality \eqref{eq:zeroes} now follows from \eqref{eq:traceformula}, \eqref{eq:tf0}, \eqref{eq:tf1}, \eqref{eq:tf2} and \eqref{eq:tf3}.
\end{proof}

\begin{corollary}\label{zeroescor}
Let $\eta<0$. Let $a$ be an analytic function in $\{\im k>\eta\}$ which satisfies \eqref{eq:ass1} with $\eta$ replaced by $\eta'$ for any $\eta'>\eta$. Moreover, assume that \eqref{eq:ass2} holds for some $A\geq 0$ and $\nu>1$ with $\eta$ replaced by $\eta'$ for any $\eta'>\eta$ sufficiently close to $\eta$. Then the zeroes $k_j$ of $a$ in $\{\im k\geq 0\}$, repeated according to multiplicities, satisfy
$$
\#\{ j:\ \im k_j\geq 0 \} \leq c_\nu A |\eta|^{-\nu} \,.
$$
\end{corollary}

\begin{proof}
We apply Proposition \ref{zeroes} for every $\eta'>\eta$ sufficiently close to $\eta$ and obtain
$$
\sum_j \left( \im k_j - \eta' \right)_+ \leq c_\nu A |\eta'|^{-\nu+1}
$$
Clearly, we have
$$
\sum_j \left( \im k_j - \eta'\right)_+ \geq |\eta'|\ \#\left\{ j:\ \im k_j\geq 0 \right\} \,.
$$
The corollary follows by passing to the limit $\eta'\to\eta$ in the thus obtained inequality.
\end{proof}


\section{Traces and determinants}

We use the standard notation $\mathfrak S_p$ for the Schatten classes with exponent $1\leq p<\infty$. If $n\in\N$, $K\in\mathfrak S_n$ and $\lambda_j(K)$ denote the eigenvalues of $K$, repeated according to algebraic multiplicities, the $n$-th order regularized determinant $\det{}_n(1+K)$ is defined by
$$
\det{}_n(1+K) = \prod_j \left( \left( 1+\lambda_j(K) \right) \exp\left( \sum_{m=1}^{n-1} \frac{(-1)^m}{m} \lambda_j(K)^m \right) \right) \,.
$$
The following properties are well-known, but we include a proof for the sake of completeness.

\begin{lemma}\label{detbounds}
Let $n\in\N$.
\begin{enumerate}
\item For any $n-1\leq p\leq n$ with $p>0$ there is a $\Gamma_{n,p}$ such that
$$
\ln |\det{}_n(1+K)| \leq \Gamma_{n,p} \|K\|_p^p \,.
$$
\item For any $0\leq\theta<1$ and $0<p\leq n$ there is a $\Gamma_{n,p}(\theta)$ such that, if $\|K\|\leq\theta$, then
$$
\left| \log \det{}_n(1+K) \right| \leq \Gamma_{n,p}(\theta)\ \|K\|_p^p \,.
$$
\end{enumerate}
\end{lemma}

\begin{proof}
To prove the first assertion, let $f(z) := (1+z) \exp \left( \sum_{m=1}^{n-1} \frac{(-1)^m}{m} z^m \right)$. Then $\ln|f(z)|$ can be bounded by a constant times $|z|^n$ for small $|z|$ and by a constant times $|z|^{n-1}$ for large $|z|$. Thus, $\ln|f(z)|\leq \Gamma_{n,p}|z|^p$ for any $n-1\leq p\leq n$, and so
$$
\ln |\det{}_n(1+K)| \leq \Gamma_{n,p} \sum_j |\lambda_j(K)|^p
$$
By Weyl's inequality \cite[Thm. 1.15]{Si}, the sum on the right side does not exceed $\|K\|_p^p$.

To prove the second assertion, we note that since $|\lambda_j(K)|\leq \|K\|\leq\theta<1$, we have
\begin{align*}
\log \det{}_n(1+K) & = \sum_j \left( \log \left(1+\lambda_j(K) \right) + \sum_{m=1}^{n-1} \frac{(-1)^m}{m} \lambda_j(K)^m \right) \\
& = \sum_j \sum_{m=n}^\infty \frac{(-1)^{m-1}}{m} \lambda_j(K)^m \,.
\end{align*}
We bound
$$
\left| \log \det{}_n(1+K) \right| \leq \sum_j \sum_{m=n}^\infty \frac{1}{m} |\lambda_j(K)|^n \theta^{m-n} = \gamma_n(\theta) \sum_j |\lambda_j(K)|^n
$$
and obtain the assertion for $p=n$ with $\Gamma_{n,n}(\theta)=\gamma_n(\theta)$ again by Weyl's inequality. If $0<p<n$, we simply use $|\lambda_j(K)|^n \leq \theta^{n-p} |\lambda_j(K)|^p$ and get the inequality with $\Gamma_{n,p}(\theta) = \theta^{n-p} \gamma_n(\theta)$.
\end{proof}

The previous proof and a simple computation show that for $n=p=2$ one can take
\begin{equation}
\label{eq:consths}
\Gamma_{2,2} = 1/2 \,.
\end{equation}

We next recall a version of the Birman--Schwinger principle. We state it in the setting of \cite{Fr3}, namely, where $H_0$ is a non-negative self-adjoint operator and $G_0$ and $G$ are operators with $\dom G\supset\dom H_0^{1/2}$ and $\dom G_0\supset\dom H_0^{1/2}$ and such that $G_0(H_0+1)^{-1/2}$ and $G(H+1)^{-1/2}$ are compact. Then the quadratic form
$$
\| H_0^{1/2} u \|^2 + (G u ,G_0 u)
$$
defines an $m$-sectorial operator, which we shall denote by $H$. The Birman--Schwinger principle states that $z\in\rho(H_0)$ is an eigenvalue of $H$ iff $-1$ is an eigenvalue of the Birman--Schwinger operator $G_0(H_0-z)^{-1} G^*$. Moreover, the corresponding geometric multiplicities coincide.

The following lemma says that even the algebraic multiplicities of eigenvalues of $H$ can be characterizes in terms of a quantity related to the Birman--Schwinger operator.

\begin{lemma}\label{mult}
Assume that for some $n\in\N$, $G_0(H_0-\zeta)^{-1}G^*\in\mathfrak S_n$ for all $\zeta\in\rho(H_0)$. Then the function $\zeta\mapsto\det{}_n(1+G_0(H_0-\zeta)^{-1}G)$ is analytic in $\rho(H_0)$. For $z\in\rho(H_0)$ one has $\det{}_n(1+G_0(H_0-z)^{-1}G)=0$ iff $z$ is an eigenvalue of $H$ and the order of the zero coincides with the algebraic multiplicity.
\end{lemma}

The analyticity of the function $\zeta\mapsto\det{}_n(1+G_0(H_0-\zeta)^{-1}G)$ is well-known and so is the result concerning the algebraic multiplicity in the case $n=1$. The result for general $n$ is essentially due to \cite{LaSu}; see also \cite{Fr3} for an extension of their proof to the present setting.


\section{Resolvent bounds}\label{sec:resbounds}

In this section we collect trace ideal bounds for the Birman--Schwinger operator
\begin{equation}
\label{eq:bs}
K(k) = \sqrt{V} (-\Delta -k^2)^{-1} \sqrt{|V|} \,.
\end{equation}
We use the notation $\sqrt{V(x)} = V(x)/\sqrt{|V(x)|}$ if $V(x)\neq 0$ and $\sqrt{V(x)}=0$ if $V(x)=0$.

We begin with the case of the half-line, that is, $-\Delta$ in \eqref{eq:bs} denotes the Dirichlet Laplacian on $(0,\infty)$. From the explicit expression of its integral kernel it is easy to see that, if $V$ is bounded and has compact support, $K(k)$ admits an analytic continuation to an entire operator family on $L^2(\R_+)$. The following proposition gives a bound on the Hilbert--Schmidt norm.

\begin{proposition}\label{resbound1d}
For any $k\in\C\setminus\{0\}$,
$$
\| K(k) \|_{\mathfrak S_2} \leq \frac{1}{|k|} \int_0^\infty e^{2x(\im k)_-} |V(x)|\,dx \,,
$$
in the sense that $K(k)$ is Hilbert--Schmidt if the integral on the right side is finite.
\end{proposition}

\begin{proof}
The integral kernel of $(-\Delta-k^2)^{-1}$ is the function
$$
g_k(x,y) = \frac{1}{2ik} \left( e^{ik(x+y)} - e^{ik|x-y|}\right) \,,
$$
which satisfies
$$
|g_k(x,y)| \leq \frac{1}{|k|} e^{(x+y)(\im k)_-} \,.
$$
Combining this bound with the identity
$$
\| K(k) \|_{\mathfrak S_2}^2 = \int_0^\infty \int_0^\infty |V(x)| |g_k(x,y)|^2 |V(y)| \,dx\,dy
$$
we obtain the claimed bound.
\end{proof}

We now consider the case of $\R^d$ with $d\geq 3$ odd. The operator $-\Delta$ in \eqref{eq:bs} denotes the Laplacian in $\R^d$. It is well-known (see, e.g., \cite[Theorem 3.1]{DZ} for a textbook proof) that, since $d$ is odd, $(-\Delta-k^2)^{-1}$ admits an analytic continuation to an entire operator family when considered as an operator from compactly supported functions in $L^2(\R^d)$ to $L^2_{\loc}(\R^d)$. Thus, if $V$ is bounded and compactly supported, $K(k)$ has a analytic continuation to an entire operator family on $L^2(\R^d)$. The following proposition implies, in particular, that if $V$ decays exponentially, then the Birman--Schwinger operator also admits an analytic continuation to (part of) the lower half-plane and that this continuation belongs to a certain trace ideal.

\begin{proposition}\label{resboundoddd}
Let $d\geq 3$ be odd. There are constants $C_d>0$, $\beta_d>0$ such that for any $k\in\C\setminus\{0\}$,
$$
\| K(k) \|_{\mathfrak S_{d+1}} \leq C_d \left( \frac{1}{|k|} \int_{\R^d} e^{\beta_d |x|(\im k)_-} |V(x)|^{(d+1)/2} \,dx \right)^{2/(d+1)} \,,
$$
in the sense that $K(k)\in\mathfrak S_{d+1}$ if the integral on the right side is finite. 
\end{proposition}

The proof of this proposition is somewhat involved and, in fact, presents the technical main result of this paper. In order to present the idea behind the proofs of Theorems \ref{main1} and \ref{main2} more clearly, we defer the proof of Proposition \ref{resboundoddd} to Section \ref{sec:resboundproof}.


\section{Proof of Theorems \ref{main1} and \ref{main2}}

In this section we prove our main results, Theorems \ref{main1} and \ref{main2}. We prove them simultaneously. Let us assume that $V$ is bounded and has compact support. The bound in this case implies the bound in the general case by a simple density argument.

As discussed in Section \ref{sec:resbounds}, the Birman--Schwinger operators $K(k)$ from \eqref{eq:bs} (with $-\Delta$ denoting the Dirichlet Laplacian if $d=1$ and the ordinary Laplacian if $d\geq 3$) extend to an entire family of bounded operators. The same proof shows that they are not only entire with respect to the norm of bounded operators, but even with respect to the norm of operators in $\mathfrak S_{d+1}$. (In fact, even in $\mathfrak S_{p}$ with $p>d/2$, see \cite[Lemma 3.21]{DZ}.) We emphasize that at this point we use the restriction to bounded, compactly supported potentials; in the general case, Propositions \ref{resbound1d} and \ref{resboundoddd} do not allow us to exclude a singularity at the origin.

We will apply Corollary \ref{zeroescor} to the function
$$
a(k) := \det{}_{d+1}(1+K(k))
$$
with $\eta= - \epsilon/\beta_d$, where $\beta_d$ is from Proposition \ref{resboundoddd} if $d\geq 3$ is odd and $\beta_1=2$ if $d=1$. Since $K(k)$ is analytic with values in $\mathfrak S_{d+1}$, the function $a$ is analytic. It follows from the resolvent bounds in Propositions \ref{resbound1d} and \ref{resboundoddd}, combined with item (2) in Lemma \ref{detbounds} (with $p=n=d+1$), that assumption \eqref{eq:ass1} is valid. Moreover, combining them with item (1) in Lemma \ref{detbounds} (again with $p=n=d+1$), we see that assumption \eqref{eq:ass2} holds with $\nu=2$ and
$$
A = \Gamma_{d+1,d+1} C_d^{d+1} \left( \int e^{\epsilon |x|} |V(x)|^{(d+1)/2}\,dx \right)^2 \,.
$$
Here $C_d=1$ if $d=1$. Thus, Corollary \ref{zeroescor} implies that
$$
\#\{ j:\ \im k_j \geq 0\} \leq \epsilon^{-2} \beta_d^2 c_2 \Gamma_{d+1,d+1} C_d^{d+1} \left( \int e^{\epsilon |x|} |V(x)|^{(d+1)/2}\,dx \right)^2 \,.
$$
It remains to use Lemma \ref{mult}, which says that the $k_j$ with $\im k_j> 0$ coincide with the square roots of the eigenvalues of $-\Delta+V$, counting algebraic multiplicities. This proves Theorems \ref{main1} and \ref{main2}.

In the case $d=1$, we can use the values of the constants
$$
c_2 = 1/2 \,,\quad \Gamma_{2,2} = 1/2 \,,\quad \beta_1 = 2 \,,\quad C_1 = 1
$$
(see \eqref{eq:consths}, \eqref{eq:zeroesconst1} and Proposition \ref{resbound1d}) to get the explicit constant in Theorem~\ref{main1}.
\qed


\section{Proof of Proposition \ref{resboundoddd}}\label{sec:resboundproof}

The bound from Proposition \ref{resboundoddd} for $\im k\geq 0$ is contained in \cite{FrSa}, so we focus on the case $\im k<0$. We are going to break the proof into two steps, according whether $|\im k|/|\re k|$ is large or small.

\begin{lemma}\label{resbound1}
Let $d\geq 3$. There are $\alpha_d,\beta_d,\gamma_d>0$ such that, for any $k\in\C$ with $\im k< 0$ and $|\re k|> \gamma_d|\im k|$,
\begin{equation}
\label{eq:resbound1}
\left\| K(k) \right\|_{\mathfrak S_{d+1}} \leq \alpha_d  \left( \frac{\int_{\R^d} e^{\beta_d |\im k||x|} |V(x)|^{(d+1)/2}\,dx}{|\re k| - \gamma_d |\im k|} \right)^{2/(d+1)} \,.
\end{equation}
One can take $\beta_d =2(e^{(d+1)/2}-1)/(e-1)$ and $\gamma_d=e^{(d+1)/2}/(e-1)$.
\end{lemma}

We emphasize that this lemma does not need $d$ to be odd. In this case one can still prove that $K(k)$ has an analytic continuation to the set $\C\setminus(-i)[0,\infty)$.

\begin{proof}
By a density argument we may assume that $V$ is bounded and compactly supported. As discussed above, under this assumption $K(k)$ has an analytic continuation to an entire function. We will show that for any $k$ as in the lemma and any finite rank operator $Q$,
\begin{equation}
\label{eq:resbound1proof}
\left| \tr K(k)Q\right| \leq \alpha_d  \left( \frac{\int_{\R^d} e^{\beta_d (\im k)_-|x|} |V(x)|^{(d+1)/2}\,dx}{|\re k| - \gamma_d |\im k|} \right)^{2/(d+1)} \|Q\|_{\mathfrak{S}_{(d+1)/d}} \,,
\end{equation}
which will imply the assertion. 

To prove \eqref{eq:resbound1proof} we use complex interpolation. Namely, for fixed $k$ as in the lemma we will construct an analytic family of operators $K_\zeta$ such that $K_1=K(k)$ for $\zeta=1$. The construction of $K_\zeta$ proceeds as follows. If $\im k>0$ and $\re\zeta\geq 0$, then the operator $(-\Delta-k^2)^{-\zeta}$ is well-defined by the spectral theorem or, equivalently, as a multiplier in Fourier space. Here $(\cdot)^{-\zeta}$ denotes the principal branch. If $\re\zeta>0$, this is an integral operator with integral kernel
\begin{equation}
\label{eq:reskernel}
\left(-\Delta-k^2 \right)^{-\zeta}(x,y) = \frac{1}{(2\pi)^d} \int_{\R^d} \frac{e^{i\xi\cdot (x-y)}}{\left(\xi^2-k^2\right)^\zeta}\,d\xi \,.
\end{equation}
We recall \cite[Section III.2.8]{GS} the formula
\begin{equation}
\label{eq:fouriertrafo}
\int_{\R^d} \frac{e^{i\xi\cdot x}}{\left(\xi^2-k^2\right)^\zeta}\,d\xi 
= (2\pi)^{d/2} \, \frac{2^{1-\zeta}}{\Gamma(\zeta)} \left( \frac{-ik}{|x|} \right)^{(d-2\zeta)/2} K_{(d-2\zeta)/2}(-ik|x|) \,,
\end{equation}
valid for $\im k>0$ and $\re\zeta>0$. Here $K_\nu$ denotes the Bessel function of the third kind. We will need the fact that
\begin{equation}
\label{eq:besselnu}
K_\nu(z) = K_{-\nu}(z) \,,
\end{equation}
as well as the following integral representation for this function \cite[Section 7.3.4]{EMOT},
\begin{equation}
\label{eq:bessel}
K_\nu(z) = \frac{1}{\Gamma(\nu+1/2)} \left(\frac{\pi}{2z}\right)^{1/2} e^{-z} \int_0^\infty e^{-t} t^{\nu-1/2} \left( 1+ \frac{t}{2z}\right)^{\nu-1/2}\,dt
\qquad\text{if}\ \re\nu>-1/2 \,.
\end{equation}
For fixed $\zeta$ with $\re\zeta>0$, the right side of \eqref{eq:fouriertrafo} has an analytic continuation with respect to $k$, with $k=0$ being possibly a branch point. This allows us to analytically continue the operator $W(-\Delta-k^2)^{-\zeta}W$ to the lower half-plane if $W$ is a bounded, compactly supported function. At the same time, for fixed $k\neq 0$ (possibly in the lower half-plane), the operator family $W(-\Delta-k^2)^{-\zeta}W$ is analytic with respect to $\zeta$ in the upper half-plane.

We now fix $k\in\C$ with $\im k< 0$ and $\re k\neq 0$ and set
$$
k_\zeta = k + i |\im k| \frac{e-e^\zeta}{e-1} \,.
$$
For fixed $\re\zeta$ this describes a circle centered at $k+i|\im k|e/(e-1)$ with radius $|\im k| e^{\re\zeta}/(e-1)$. 

We consider the function
$$
f(\zeta):= e^{\zeta^2} \, \tr \left( S |V|^{\zeta/2} (-\Delta-k_\zeta^2)^{-\zeta} |V|^{\zeta/2} U |Q|^{(d+1-\zeta)/d} \right) \,,
$$
where $S(x)=V(x)/|V(x)|$ if $V(x)\neq 0$ and $S(x)=0$ if $V(x)=0$ and where $Q=U|Q|$ is the polar decomposition of $Q$. The function $f$ is analytic in $\zeta$ and at $\zeta=1$ its absolute value coincides with $e$ times the left side of \eqref{eq:resbound1proof}. We will apply Hadamard's three lines lemma to the function $f$, where the bounding lines are given by $\re\zeta=0$ and $\re\zeta=(d+1)/2$.

If $\re\zeta=0$, we use the fact that $\im k_\zeta\geq 0$. This implies that the argument of $|\xi|^2 - k_\zeta^2$ is uniformly bounded in $\xi\in\R^d$ and therefore
$$
\left\| (-\Delta-k_\zeta^2)^{-\zeta} \right\| 
= \sup_{\xi\in\R^d} \left| (|\xi|^2-k_\zeta^2)^{-\zeta} \right| \leq C_1 e^{C_1 |\im\zeta|}
\qquad\text{if}\ \re\zeta=0 \,.
$$
for some $C_1>0$. Thus, because of the superexponential decrease of the factor $e^{-(\im\zeta)^2}$,
\begin{equation}
\label{eq:hadamard1}
|f(\zeta)| \leq C_1' \|Q\|_{(d+1)/d}^{(d+1)/d} 
\qquad\text{if}\ \re\zeta=0 \,.
\end{equation}

If $\re\zeta=(d+1)/2$, we bound
\begin{align*}
|f(\zeta)| & \leq e^{(\re\zeta)^2-(\im\zeta)^2} \, \left\| |V|^{(d+1)/4} (-\Delta-k_\zeta^2)^{-\zeta} |V|^{(d+1)/4} \right\|_{\mathfrak S_2} \left\| |Q|^{(d+1)/(2d)} \right\|_{\mathfrak S_2} \\
& = e^{(\re\zeta)^2-(\im\zeta)^2} \, \left\| |V|^{(d+1)/4} (-\Delta-k_\zeta^2)^{-\zeta} |V|^{(d+1)/4} \right\|_{\mathfrak S_2} \left\|Q \right\|_{\mathfrak S_{(d+1)/d}}^{(d+1)/(2d)} \,.
\end{align*}
In order to control the Hilbert--Schmidt norm on the right side, we bound the integral kernel of $(-\Delta-k^2)^{-\zeta}$ for $\zeta=(d+1)/2+i\tau$ with $\tau\in\R$. According to \eqref{eq:reskernel}, \eqref{eq:fouriertrafo} and \eqref{eq:besselnu} it is given by
$$
(2\pi)^{-d/2} \, \frac{2^{1-(d+1)/2 -i\tau}}{\Gamma((d+1)/2+i\tau)} \left( \frac{-ik}{|x-y|} \right)^{-1/2-i\tau} K_{1/2+i\tau}(-ik|x-y|) \,.
$$
We take $k=k_\zeta$ and bound, using \eqref{eq:bessel},
\begin{align*}
& \left| K_{1/2+i\tau}(-ik_\zeta|x-y|) \right| \\
& \qquad\qquad \leq \left(\frac{\pi}{2|k_\zeta||x-y|}\right)^{1/2} \frac{e^{-\im k_\zeta|x-y|}}{|\Gamma(1+i\tau)|} \int_0^\infty e^{-t} \left| \left(1+ \frac{it}{2k_\zeta|x-y|}\right)^{i\tau} \right| dt \,.
\end{align*}
It is easy to see that there is a constant $C>0$ such that $|\re k_\zeta| \geq C |\im k_\zeta|$ for all $\zeta$ with $\re\zeta=(d+1)/2$. This implies that the argument of $1+it/(2k_\zeta|x-y|)$ is uniformly bounded in $t\in [0,\infty)$, $|x-y|\in [0,\infty)$ and $\tau=\im\zeta\in\R$. We now observe that
$$
\im k_\zeta = -|\im k| \frac{e^{(d+1)/2}\cos\tau-1}{e-1} \geq - \beta_d |\im k|/2
$$
with $\beta_d= 2(e^{(d+1)/2}-1)/(e-1)$ and
$$
|k_\zeta| \geq \sqrt{(\re k)^2 + \frac{(\im k)^2}{(e-1)^2}} - |\im k| \frac{e^{\re \zeta}}{e-1} \geq |\re k| - \gamma_d |\im k|
$$
with $\gamma_d = e^{(d+1)/2}/(e-1)$. Thus we conclude that
$$
\left| (-\Delta-k_\zeta^2)^{-1}(x,y) \right| \leq C_2 e^{C_2|\tau|} \frac{e^{\beta_d |\im k| |x-y|/2}}{|\re k|-\gamma_d |\im k|} \,,
\qquad
\zeta = \frac{d+1}{2} + i\tau \,,
$$
for some constant $C_2$ depending on $d$, but not on $x,y$ or $\tau$. This implies
$$
\left\| |V|^{(d+1)/4} (-\Delta-k_\zeta^2)^{-\zeta} |V|^{(d+1)/4} \right\|_{\mathfrak S_2}^2 \leq C_2^2 e^{2C_2|\tau|} \left( \frac{ \int_{\R^d} e^{\beta_d |\im k| |x|} |V(x)|^{(d+1)/2} \,dx }{|\re k|-\gamma_d |\im k|} \right)^2
$$
and therefore,
\begin{equation}
\label{eq:hadamard2}
|f(\zeta)| \leq C_2' \|Q\|_{\mathfrak S_{(d+1)/d}}^{(d+1)/(2d)} \frac{ \int_{\R^d} e^{\beta_d |\im k| |x|} |V(x)|^{(d+1)/2} \,dx }{|\re k|-\gamma_d |\im k|} 
\qquad\text{if} \ \re\zeta = \frac{d+1}{2} \,.
\end{equation}
According to Hadamard's three lines lemma we have
$$
|f(1)|\leq \left( \sup_{\re\zeta=0}|f(\zeta)| \right)^{(d-1)/(d+1)} \left( \sup_{\re\zeta=(d+1)/2}|f(\zeta)| \right)^{2/(d+1)} \,.
$$
Combining this with the bounds \eqref{eq:hadamard1} and \eqref{eq:hadamard2} we obtain \eqref{eq:resbound1proof}
\end{proof}

\begin{lemma}\label{resbound2}
Let $d\geq 3$ be odd. There is a constant $\alpha_d'$ such that for any $k\in\C$ with $\im k<0$,
\begin{equation}
\label{eq:resbound2}
\left\| K(k) \right\|_{\mathfrak S_{d+1}} \leq \alpha_d'  \left( \frac{\int_{\R^d} e^{(d+1) |k||x|} |V(x)|^{(d+1)/2}\,dx}{|k|} \right)^{2/(d+1)} \,.
\end{equation}
\end{lemma}

\begin{proof}
Since $d$ is odd,
\begin{align*}
\left\| K(k) \right\|_{\mathfrak S_{d+1}}^{d+1} & = \tr \left( K(k)^* K(k) \cdots K^*(k) K(k) \right) \\
& = \int\cdots\int |V(x_1)| \overline{g_k(x_{d+1},x_1} |V(x_{d+1})| g_k(x_{d+1},x_d) \cdots   g_k(x_4,x_3) \\
& \qquad\qquad \qquad \times |V(x_3)| \overline{g_k(x_2,x_3)} |V(x_2)| g_k(x_2,x_1) \,dx_1\cdots dx_{d+1} \,,
\end{align*}
where $g_k(x,y)$ is the integral kernel of the operator $(-\Delta-k^2)^{-1}$. It follows from formula \eqref{eq:bessel} for the Bessel function that
$$
\left| K_\nu(-ia) \right| \leq e^{2|a|} K_\nu(|a|)
\qquad\text{if}\ \im a \leq 0 \,.
$$
This implies
$$
|g_k(x,y)|\leq e^{2|k||x-y|} g_{i|k|}(x,y) \leq e^{2|k|(|x|+|y|)} g_{i|k|}(x,y) \,,
$$
and therefore, in view of the above expression for $\left\| K(k) \right\|_{\mathfrak S_{d+1}}^{d+1}$,
\begin{align*}
\left\| K(k) \right\|_{\mathfrak S_{d+1}}^{d+1}
& = \int\cdots\int e^{2|k||x_1|} |V(x_1)| \, g_{i|k|}(x_{d+1},x_1) \, e^{2|k||x_{d+1}|} |V(x_{d+1})| \, g_{i|k|}(x_{d+1},x_d) \cdots \\
& \qquad\qquad \qquad \times g_{i|k|}(x_4,x_3)\, e^{2|k||x_3|} |V(x_3)| \, g_{i|k|}(x_2,x_3) \, e^{2|k|x_2|} |V(x_2)| \, \\
& \qquad\qquad \qquad \times g_{i|k|}(x_2,x_1) \,dx_1\cdots dx_{d+1} \\
& = \left\| e^{2|k||x|} K(i|k|) e^{2|k||y|} \right\|_{\mathfrak S_{d+1}}^{d+1} \,,
\end{align*}
To bound the right side we use the Kato--Seiler--Simon bound \cite[Thm. 4.1]{Si} and get
\begin{align*}
& \left\| e^{|k||x|} K(i|k|) e^{|k||y|} \right\|_{\mathfrak S_{d+1}}^{(d+1)/2}
\leq \left\| e^{|k||x|} K(i|k|) e^{|k||y|} \right\|_{\mathfrak S_{(d+1)/2}}^{(d+1)/2} \\
& \qquad\qquad = \left\| (-\Delta+|k|^2)^{-1/2} \sqrt{|V|} e^{|k||x|} \right\|_{\mathfrak S_{d+1}}^{d+1} \\
& \qquad\qquad\leq (2\pi)^{-d} \int_{\R^d} \frac{d\xi}{(|\xi|^2+|k|^2)^{(d+1)/2}} \int_{\R^d} |V(x)|^{(d+1)/2} e^{(d+1)|k||x|} \,dx \\
& \qquad\qquad = (2\pi)^{-d} \int_{\R^d} \frac{d\xi}{(1+|\xi|^2)^{(d+1)/2}} |k|^{-1} \int_{\R^d} |V(x)|^{(d+1)/2} e^{(d+1)|k||x|} \,dx \,.
\end{align*}
This proves the lemma.
\end{proof}

Finally, we are in position to give the

\begin{proof}[Proof of Proposition \ref{resboundoddd}]
The claimed bound for $\im k\geq 0$ follows from \cite{FrSa}. The bound for $\im k<0$ and $|\re k|\geq 2\gamma_d|\im k|$ follows from Lemma \ref{resbound1}. (Note that in this case one can bound $|\re k|-\gamma_d|\im k|\geq (\gamma_d/\sqrt{1+4\gamma_d^2}) |k|$ in the denominator.) Finally, the bound for $\im k<0$ and $|\re k|< 2\gamma_d|\im k|$ follows from Lemma \ref{resbound2}. (Note that in this case one can bound $|k|\leq \sqrt{1+4\gamma_d^2} |\im k|$ in the exponential.) This concludes the proof.
\end{proof}



\bibliographystyle{amsalpha}

\end{document}